\documentclass[11pt,leqno]{article}
\usepackage{geometry}
\newgeometry{vmargin={28mm}, hmargin={35mm,35mm}}   

\usepackage{xcolor}

\usepackage[hidelinks]{hyperref}

\usepackage{amsmath,amsthm,amsfonts,amssymb,latexsym,amscd,enumerate}
\usepackage{slashed}

\usepackage{todonotes}

\usepackage{palatino}
\usepackage[mathcal]{euler}
\usepackage{tikz-cd}
\usepackage{graphicx}

\allowdisplaybreaks

\usepackage{xy}
\xyoption{all}

\numberwithin{equation}{section}
\swapnumbers

\newtheorem{theorem}{Theorem}[section]
\newtheorem*{theorem*}{Theorem}

\newtheorem{proposition}[theorem]{Proposition}
\newtheorem*{proposition*}{Proposition}
\newtheorem{lemma}[theorem]{Lemma}

\theoremstyle{definition}
\newtheorem{definition}[theorem]{Definition}

\newtheorem{example}[theorem]{Example}

\newtheorem{remark}[theorem]{Remark}

\DeclareMathOperator{\End}{End}

\DeclareMathOperator{\Cliff}{Cliff}

\DeclareMathOperator{\CliffordOrder}{Clifford-order}

\DeclareMathOperator{\str}{str}
\DeclareMathOperator{\Str}{Str}
\DeclareMathOperator{\cosech}{cosech}

\newcommand{\R}{\mathbb{R}}
\newcommand{\Z}{\mathbb{Z}}
\newcommand{\N}{\mathbb{N}}
\newcommand{\C}{\mathbb{C}}
\newcommand{\T}{\mathbb{T}}

\newcommand{\B}{\mathsf{B}}

\begin{document}

	\title{A Short Proof of the Localization Formula for the Loop Space Chern Character of Spin Manifolds}
	
	\author{Matthias Ludewig\footnote{Universit\"at Regensburg} ~and Zelin Yi\footnote{Chern Institute of Mathematics, Nankai University}}
	
	\date{}
	
	\maketitle
	
	\begin{abstract}
	 In this note, we give a short proof of the localization formula for the loop space Chern character of a compact Riemannian spin manifold $M$, using the rescaled spinor bundle on the tangent groupoid associated to $M$.
	\end{abstract}

\section*{Introduction}

In supersymmetric quantum mechanics, one is interested in the path integral corresponding to the $N=1/2$ supersymmetric $\sigma$-model, which mathematically can be viewed as an integration functional for differential forms $\xi$ on the loop space $\mathsf{L}M$ of a compact Riemannian spin manifold $M$. Formally, this should be given by the expression
\begin{equation} \label{PathIntegralFormula}
  I[\xi] = \int_{\mathsf{L} M} e^{-S + \omega} \wedge \xi
\end{equation}
where $S$ is the classical energy and $\omega$ is a certain canonical 2-form on the loop space. This formula and its close relation to the Atiyah-Singer index theorem has sparked interest for more than 30 years (see \cite{AtiyahCircular, AlvarezGaume2} and the introduction of \cite{BatuMatthias19} for further references).

The path integral formula \eqref{PathIntegralFormula} has been finally given a rigorous interpretation for a certain class of differential forms $\theta$ in \cite{hanisch2017rigorous, BatuMatthias19}; see also \cite{ludewig2019construction}. Essentially, this is the class of \emph{iterated integrals}, first considered by Chen \cite{Chen1} and later extended by Getzler, Jones and Petrack \cite{GJP} in order to contain the Bismut-Chern character forms first introduced by Bismut \cite{BismutDH}. 
Iterated integrals are the image of the \emph{iterated integral map}, which maps the cyclic chain complex associated to the algebra $\Omega(M)$ of differential forms to the algebra $\Omega(\mathsf{L}M)$ of differential forms on the loop space.

\medskip 

Pulling back the integration functional $I$ of \cite{hanisch2017rigorous} with the iterated integral map, one obtains a coclosed functional on the cyclic chain complex of $\Omega(M)$, which we denote by $\mathrm{Ch}_D$; namely, it has then been observed in \cite{BatuMatthias19} that this functional can be viewed as a non-commutative \emph{Chern character} associated to a Fredholm module over $\Omega(M)$ determined by the Dirac operator $D$ on $M$, with a combinatorial formula similar to the JLO cocycle \cite{JLO}. 

The most remarkable feature of the loop space path integral $I$ and its combinatorial counterpart $\mathrm{Ch}_D$ is that they satisfy a \emph{localization formula} of Duistermaat-Heckmann type, as though the loop space $\mathsf{L}M$ was a finite-dimensional manifold; see \cite[Thm.~3.19]{hanisch2017rigorous} and \cite[Thm.~E]{BatuMatthias19}. This in particular facilitates the proof of the Atiyah Singer index theorem using \emph{loop space localization} envisioned by Atiyah \cite{AtiyahCircular}. 

\begin{theorem*}
Let $M$ be a compact Riemannian spin manifold of even dimension $d$ and let $c$ be an entire cyclic chain over $\Omega(M)$. If $c$ is closed, then
  \begin{equation} \label{LocalizationFormula}
    \mathrm{Ch}_D(c) = \frac{1}{(2\pi i)^{d/2}} \int_M \hat{A}(M) \wedge i(c).
  \end{equation} 
\end{theorem*}

Here $\hat{A}(M)$ denotes the $\hat{A}$-genus-form, and $i$ is the combinatorial analog of the map that restricts a differential form on the loop space to the fixed point set $M \subset \mathsf{L}M$ (see \eqref{RestrictionMap} below). In fact, we prove formula \eqref{LocalizationFormula} more generally for entire cyclic chains over the \emph{acyclic extension} $\Omega_{\T}(M)$ of $\Omega(M)$ (see \S\ref{SectionBarComplex} below), which is necessary in order to encompass the Bismut-Chern characters.

\medskip

The purpose of this note is to give a short proof of the above theorem using Connes' tangent groupoid and its extension to the spinor bundle introduced by Higson and the second-named author \cite{HigsonYi19, ZelinYi19}. The strategy is as follows: Rescaling of the Fredholm module yields a one-parameter family of Chern-characters $\{\mathrm{Ch}_t\}_{t > 0}$, which are all cohomologous on suitable complexes by the the homotopy invariance of the Chern character. For \emph{closed} chains, formula \eqref{LocalizationFormula} therefore follows from calculating the limit of $\mathrm{Ch}_t$, as $t\rightarrow 0$; the result is Thm.~\ref{LimitTheorem} below, which can be understood as a ``loop space version'' of a corresponding result of Block-Fox on the JLO cocycle \cite[Thm.~4.1]{BlockFox}.

It is the calculation of this limit for which the machinery of the tangent groupoid is particularly useful. In brief, the Chern character is defined as the supertrace of a certain family of operators, the kernels of which turn out to patch together to define a smooth section of the \emph{rescaled spinor bundle} $\mathbb{S}$ over the tangent groupoid $\mathbb{T}M$ that \emph{can be extended down to zero}. Here one then has a one-parameter family of supertraces at disposal (defined in \cite{HigsonYi19}), which allow to easily calculate the value at zero.

\medskip

Below, we briefly explain the construction of the Chern characters $\mathrm{Ch}_t$ and reduce the proof of formula \eqref{LocalizationFormula} to the calculation of the short time limit of $\mathrm{Ch}_t$. This is essentially algebraic. In the second part of this note, we briefly recount the theory of the rescaled spinor bundle and prove the relevant Thm.~\ref{LimitTheorem} below.

\medskip

\noindent{\textbf{Acknowledgements.}} We are pleased to thank Nigel Higson for helpful discussions regarding this paper.
The first-named author acknowledges funding from ARC Discovery Project grant FL170100020 under Chief Investigator and Australian Laureate Fellow Mathai Varghese.

\section{The Chern character}

In this section, we give a short review of the construction of the Chern character associated to the Fredholm module over $\Omega(M)$ determined by the Dirac operator over a compact Riemannian spin manifold. We focus on the algebraic construction, leaving out many analytical details; for these, we refer to \cite{BatuMatthias19}.

\subsection{The bar complex} \label{SectionBarComplex}

Let $M$ be a manifold and let $\Omega(M)$ its complex of (complex-valued) differential forms. The \emph{acyclic extension} of $\Omega(M)$ is the algebra $\Omega_{\T}(M) := \Omega(M)[\sigma]$, where $\sigma$ is a formal variable of degree $-1$ satisfying $\sigma^2 = 0$. Elements of $\Omega_{\T}(M)$ will be written as $\theta = \theta^\prime + \sigma \theta^{\prime\prime}$ with $\theta^\prime, \theta^{\prime\prime} \in \Omega(M)$.
$\Omega_{\T}(M)$ has a differential $d_{\mathbb{T}} = d - \iota$, where $d$ is the de-Rham differential and $\iota$ is defined by $\iota(\theta^\prime + \sigma \theta^{\prime\prime}) = \theta^{\prime\prime}$. Since $d_{\T}$ is not homogeneous, $\Omega_{\T}(M)$ is only $\Z_2$-graded.

\begin{remark}
  In \cite{GJP}, the algebra $\Omega(M \times \T)^{\T}$ of $\T$-invariant differential forms on $M \times \T$ is used (where $\T = S^1$). This corresponds to our setup through setting $\sigma = \mathbf{t}^2 d\varphi$, where $\varphi$ is the coordinate of $\T$ and $\mathbf{t}$ is a formal variable of degree $-1$. Carrying around the formal variable $\mathbf{t}$ would allow us to stay $\Z$-graded throughout, but we feel that for this presentation, it is simpler to leave this variable out and to just take the grading mod $2$.
\end{remark}

The \emph{bar complex} of $\Omega_{\T}(M)$ is
\begin{equation*}
  \B(\Omega_{\T}(M)) = \bigoplus_{N=0}^\infty \Omega_{\T}(M)\langle1\rangle^{\otimes N},
\end{equation*}
where $\Omega_{\T}(M)\langle1\rangle$ denotes a grading shift, i.e.\ $(\Omega_{\T}(M)\langle1\rangle)^\ell = \Omega_{\T}^{\ell+1}(M)$.
There are two differentials on $\mathsf{B}(\Omega_{\T}(M))$, given on homogeneous elements by
\begin{equation*}
\begin{aligned}
b_0(\theta_1, \dots, \theta_N) &= \sum_{k=1}^N (-1)^{n_{k-1}}({\theta}_1, \dots, {\theta}_{k-1}, d_{\T} \theta_k, \dots, \theta_N)\\
b_1(\theta_1, \dots, \theta_N) &= -\sum_{k=1}^{N-1} (-1)^{n_k}({\theta}_1, \dots, {\theta}_{k-1}, {\theta}_{k} \wedge \theta_{k+1}, \theta_{k+2}, \dots, \theta_N),
\end{aligned}
\end{equation*}
where $n_k = |\theta_1|+\dots + |\theta_k|-k$. Here elements of $\B(\Omega_{\T}(M))$ are written as $(\theta_1, \dots, \theta_N)$, omitting the tensor signs for brevity. 
The above differentials satisfy $b_0 b_1 + b_1 b_0 = 0$, hence turn $\B(\Omega_{\T}(M))$ into a ($\Z_2$-graded) bi-complex with total differential $b := b_0+ b_1$. The differentials descend to the subspace
\begin{equation} \label{CyclicChains}
\mathsf{B}^{\natural}(\Omega_{\T}(M)) = \mathrm{span}~\Bigl\{ \sum_{k=0}^N (-1)^{n_k(n_N-n_k)} (\theta_{k+1}, \dots, \theta_N, \theta_1, \dots, \theta_k) \Bigr\} 
\end{equation} 
of \emph{cyclic chains}, making it a subcomplex.

\medskip

\begin{remark}
The significance of the space $\B(\Omega_{\T}(M))$ for loop space geometry is that via Chen's iterated integral map, it provides a combinatorial model for the space of equivariant differential forms on the loop space of $M$ via the \emph{(extended) iterated integral map} $\rho$ \cite{Chen1, GJP}. Explicitly, it is given by the formula
\begin{equation} \label{IteratedIntegralFormula}
  \rho(\theta_1, \dots, \theta_N) = \int_{\Delta_N} (\iota_K \theta^\prime_1(\tau_1)  + \theta^{\prime\prime}(\tau_1)) \wedge \cdots \wedge (\iota_K \theta^\prime_N(\tau_N)  + \theta^{\prime\prime}(\tau_N)) d \tau,
\end{equation}
where we wrote $\theta(\tau)$ for the pullback of $\theta \in \Omega(M)$ by the evaluation-at-$\tau$-map $\mathrm{ev}_\tau : \mathsf{L}M \rightarrow M$, $\gamma \mapsto \gamma(\tau)$ and $\iota_K$ denotes insertion of the velocity vector field $K(\gamma) = \dot{\gamma}$ on $\mathsf{L} M$. The iterated integral map can be viewed as a differential form counterpart of the Jones isomorphism \cite{JonesIso}, which connects the bar complex of the dg algebra of singular chains on $M$ to chains on the loop space.

A straightforward calculation shows that the iterated integral map sends the subspace $\mathsf{B}^{\natural}(\Omega_{\T}(M))$ of cyclic chains to the space $\Omega(\mathsf{L}M)^{\T} \subset \Omega(\mathsf{L}M)$ of equivariant differential forms (where $\T$ acts by rotation) and on this domain intertwines the differential $b = b_0 + b_1$ with the equivariant differential $d+ \iota_K$ on $\mathsf{L}M$.
\end{remark}

The Bismut-Chern characters alluded to above in fact do not live in $\mathsf{B}(\Omega_{\T}(M))$ but in the larger complex of \emph{entire chains} $\mathsf{B}_\epsilon(\Omega_{\T}(M))$, which allows certain infinite sums of chains. It is defined as the closure of $\mathsf{B}(\Omega_{\T}(M))$ with respect to the seminorms
\begin{equation} \label{EntireSeminorms}
  \epsilon_k(c) := \sum_{N=0}^\infty \frac{\|c_N\|_{k, N}}{\lfloor N/2\rfloor!} \qquad \text{for} \qquad c = \sum_{N=0}^\infty c_N.
\end{equation}
Here $c_N \in \Omega_{\T}(M)\langle 1\rangle^{\otimes N} \subset \Omega(M^N)[\sigma_1, \dots, \sigma_N]\langle N\rangle$ and $\|-\|_{k, N}$ denotes the $C^k$ norm on $\Omega(M^N)$; see \cite[\S3.3]{BatuMatthias19} for details. The differential $b$ extends to the entire complex and in total, we have the hierarchy of chain complexes
\begin{equation*}
\begin{tikzcd}[column sep=0cm, row sep = 0cm]
 \text{\footnotesize{(entire cyclic chains)}}~~~ & \mathsf{B}_\epsilon^\natural(\Omega_{\T}(M)) & \subset & \mathsf{B}_\epsilon(\Omega_{\T}(M)) &  ~~~ \text{\footnotesize{(entire chains)}}\\
 & \reflectbox{\rotatebox[origin=c]{90}{$\subset$}} & & \reflectbox{\rotatebox[origin=c]{90}{$\subset$}} &\\
 \text{\footnotesize{(cyclic chains)}} & \mathsf{B}^\natural(\Omega_{\T}(M)) & \subset & \mathsf{B}(\Omega_{\T}(M)).& ~~\text{\footnotesize{(bar chains)}}
\end{tikzcd}
\end{equation*}
The Bismut-Chern character $\mathrm{Ch}(p)$ defined below are entire cyclic chains, while the Chern character $\mathrm{Ch}_D$ (to be defined in the next section) is a linear functional, which is \emph{a priori} defined on $\mathsf{B}(\Omega_{\T}(M))$ but turns out to satisfy the necessary estimates to extend to the space of entire chains. In particular, $\mathrm{Ch}_D$ can be evaluated on $\mathrm{Ch}(p)$.

\begin{example} \label{ExampleBismutChernCharacter}
The Bismut-Chern characters forms on the loop space $\mathsf{L}M$ were first defined in \cite{BismutDH}, while there combinatorial versions, to be reviewed now, were introduced by Getzler-Jones Petrack \cite[\S6]{GJP}. Let $p$ be a smooth function on $M$ with values  in $M_n(\C)$ such that $p^2 = p$ and let $p^\perp = 1 - p$. Then $E := \mathrm{im}(p)$ is a vector bundle on $M$, which inherits a natural metric and connection from the trivial $\C^n$ bundle over $M$ (in fact, any vector bundle with connection on $M$ can be realized this way \cite[Thm.~1]{NarasimhanRamanan}). We set
  \begin{equation} \label{DefinitionR}
    \mathcal{R} := (2p-1) dp  + \sigma (dp)^2,
  \end{equation}
  which is an odd element of $\Omega_{\T}(M)$. The \emph{(cyclic) Bismut Chern character} of $p$ is then defined by
  \begin{equation} \label{DefinitionBismutChernCharacters}
    \mathrm{Ch}(p) := \sum_{N=0}^\infty (-1)^N\sum_{k=0}^N \mathrm{tr} (\underbrace{\mathcal{R}, \dots, \mathcal{R}}_{k}, \sigma p,\underbrace{\mathcal{R}, \dots, \mathcal{R}}_{N-k}),
  \end{equation}
where $\mathrm{tr}$ is the functional defined for elements $\Theta_1, \dots, \Theta_N \in \Omega_{\T}(M) \otimes M_n(\C)$ by
\begin{equation*}
  \mathrm{tr}(\Theta_1, \dots, \Theta_N) = \sum_{a_1, \dots, a_N = 1}^n ({\Theta_1}_{a_N}^{a_1}, {\Theta_2}_{a_1}^{a_2}, \dots, {\Theta_N}_{a_{N-1}}^{a_N}).
\end{equation*}
Due to the grading shift in the definition of $\mathsf{B}(\Omega_{\T}(M))$, it is an even chain, which is clearly cyclic, i.e.\ contained in the subcomplex \eqref{CyclicChains}. It was shown by \cite[\S6]{GJP} that the iterated integral map \eqref{IteratedIntegralFormula} sends $\mathrm{Ch}(p)$ to the Bismut-Chern character form on $\Omega(\mathsf{L} M)$, defined in \cite{BismutDH}.
A complication of the theory is that $\mathrm{Ch}(p)$ is \emph{not} closed with respect to the differential $b$; instead $b (\mathrm{Ch}(p))$ is contained in a certain subcomplex of $\mathsf{B}^\natural_\epsilon(\Omega_{\T}(M))$. One says that $\mathrm{Ch}(p)$ is closed as a \emph{Chen normalized cochain}, see \cite[\S7]{BatuMatthias19}. 
\end{example}

\subsection{Cochains and the Chern character} \label{SectionChernCharacters}

Given a $\Z_2$-graded algebra\footnote{In the case that $\mathcal{L} = \C$, we endow $\C$ with the trivial grading rendering it purely even and just speak of \emph{bar cochains}.} $\mathcal{L}$, an \emph{$\mathcal{L}$-valued bar cochain} over $\Omega_{\T}(M)$ is a linear map $\ell: \B(\Omega_{\T}(M)) \rightarrow \mathcal{L}$. Such a cochain can be viewed as a sequence of multilinear maps 
\begin{equation*}
\ell : \underbrace{\Omega_{\T}(M) \times \cdots \times \Omega_{\T}(M)}_N \rightarrow \mathcal{L},
\end{equation*}
again denoted by the same letter. In particular, for $N=0$, this is just an element of $\mathcal{L}$, which we denote by $\ell(\emptyset)$, by abuse of notation.
We say that $\ell$ is \emph{even} if it preserves parity and \emph{odd} if it reverses parity.
The standard coalgebra structure on the tensor algebra $\B(\Omega_{\T}(M))$ induces a product on bar cochains, given by
\begin{equation} \label{ProductOfCochains}
 (\ell_1 \ell_2)(\theta_1, \dots, \theta_N) = \sum_{k=0}^N (-1)^{n_k |\ell_2|} \ell_1(\theta_1, \dots, \theta_k) \ell_2(\theta_{k+1}, \dots, \theta_N),
\end{equation}
where $n_k = |\theta_1|+\dots + |\theta_k|-k$.
This product is compatible with the codifferential $\beta$ on cochains defined by
\begin{equation*}
  (\beta\ell)(\theta_1, \dots, \theta_N) = - (-1)^{|\ell|} \ell(b(\theta_1, \dots, \theta_N),
\end{equation*}
in the sense that $\beta(\ell_1 \ell_2) = \beta(\ell_1) \ell_2 + (-1)^{|\ell_1|} \ell_1 \beta(\ell_2)$ for all homogeneous cochains $\ell_1, \ell_2$; in other words, $\beta$ is a derivation on the cochain algebra.

\medskip

To obtain interesting cochains on $\mathsf{B}(\Omega_{\T }(M))$, one starts with a \emph{Fredholm module} over $\Omega(M)$, which is a triple $(H, Q, c)$, consisting of a $\Z_2$-graded Hilbert space $H$, an odd operator $Q$ on $H$ and an even linear map $c: \Omega(M) \rightarrow B(H)$, which are assumed to satisfy 
\begin{equation} \label{PropertiesFredholm}
[Q, c(f)] = c(df)], \quad \text{and} \quad c(f\theta) = c(f) c(\theta)
\end{equation}
 for $f \in \Omega^0(M)$ and $\theta \in \Omega(M)$, together with some further analytic conditions; see \cite[\S2]{BatuMatthias19} for details. 
 
 \begin{example} \label{ExampleFredholmModule}
 Our main example here is defined in case that $M$ is an even-dimensional compact spin manifold with spinor bundle $S$; in that case, $H = L^2(M, S)$, $Q = D$, the Dirac operator, and $c$ is the quantization map (see \S\ref{SectionScalingOrder} below).
 \end{example}

Inspired by Quillen \cite{Quillen}, a Fredholm module induces a \emph{connection} $\omega$, which is a cochain on $\B(\Omega_{\T}(M))$ with values in linear operators on $H$, given by
\begin{equation*}
   \omega(\emptyset) = - Q, \qquad \omega(\theta) = c(\theta^\prime), \qquad \omega(\theta_1, \dots, \theta_k) = 0  ~~(k\geq 2).
\end{equation*} 
Due to the grading shift in the definition of $\B(\Omega_{\T}(M))$, this is an odd cochain. Its \emph{curvature} is defined by the formula $F := \beta \omega + \omega^2$. Explicitly,  its components can be easily calculated to be $F(\emptyset) = Q^2$, 
\begin{equation}\label{F-operator}
\begin{split}
F(\theta) &= c(d\theta^\prime) - [Q, c(\theta^\prime)] - c(\theta^{\prime\prime}), \\
F(\theta_1,\theta_{2}) &= (-1)^{|\theta_1|} \left(c(\theta_1^\prime)c(\theta_{2}^\prime)-c(\theta_1^\prime \wedge \theta_{2}^\prime)\right)
\end{split}
\end{equation}
and $F(\theta_1, \dots, \theta_k) = 0$ for $k \geq 3$; here $[Q, c(\theta^\prime)]$ denotes the graded commutator. Motivated by the definition of the Chern-character from Chern-Weil theory, one now makes the following definition.

\begin{definition} \label{DefinitionChernCharacter}
The \emph{Chern character} of a Fredholm module $(H, Q, c)$ is the bar cochain
\begin{equation} \label{DefinitionChern0}
  \mathrm{Ch}_Q = \Str(e^{-F}).
\end{equation}
\end{definition}

Here, because $F$ takes values in \emph{unbounded} operators on $H$ and some care is needed to make sense of the exponential. This is dealt with by writing $F = Q^2 + F^\prime$ and expanding $e^{-F} = e^{-Q^2 - F^\prime}$ as a perturbation series. This results in the formula
\begin{equation} \label{PerturbationSeries}
  \mathrm{Ch}_Q = \sum_{N=0}^\infty (-1)^N \int_{\Delta_N} \Str\Bigl(e^{-\tau_1 Q^2} \prod_{p=1}^N F^\prime e^{-(\tau_{p+1}-\tau_p)Q^2} \Bigr)d \tau,
\end{equation}
where $\Delta_N = \{0 \leq \tau_1 \leq \dots \leq \tau_N \leq \tau_{N+1} := 1\}$ is the $N$-simplex, giving an explanation to the right hand side of \eqref{DefinitionChern0}.

\begin{proposition}
Restricted to the space of cyclic chains, we have $\beta (\mathrm{Ch}_Q) = 0$.
\end{proposition}

\begin{proof}[Proof (Sketch)]
Since $\beta$ is a derivation with respect to the product \eqref{ProductOfCochains}, the curvature $F$ satisfies the Bianchi identity 
\begin{equation*}
\beta F = \beta^2 \omega + \beta(\omega^2) = (\beta \omega) \omega - \omega (\beta \omega) = F\omega - \omega F = [F, \omega].
\end{equation*}
Hence
  \begin{equation*}
     \beta(\mathrm{Ch}) = \Str( \beta(e^{-F})) = - \Str( e^{-F} \beta F) = - \Str(e^{-F}[F, \omega]) = - \Str([e^{-F}, \omega]).
  \end{equation*}
One now verifies that the right hand side is zero when restricted to cyclic chains to finish the proof. The above calculations are somewhat formal and the proof remains a sketch here due to the analytical difficulties involved in defining $e^{-F}$; a complete treatment is given in \cite[\S4]{BatuMatthias19}.
\end{proof}

In order to evaluate $\mathrm{Ch}_Q$ at the Bismut-Chern characters \eqref{DefinitionBismutChernCharacters}, one needs the following proposition; see \cite[Thm.~B]{BatuMatthias19}.

\begin{proposition} \label{PropAnalyticityChern}
The Chern character is \emph{analytic}, i.e. continuous with respect to the seminorms \eqref{EntireSeminorms} and hence extends to a cochain on the entire complex $\mathsf{B}_\epsilon(\Omega_{\T}(M))$.
\end{proposition}

Before we give an explicit formula for the components of $\mathrm{Ch}_D$, we slightly generalize our Example~\ref{ExampleFredholmModule}. Namely, given a parameter $t >0$, one can define a new Fredholm module $(H, Q_t, c_t)$ by setting $Q_t = t D$ and $c_t(\theta) = t^{|\theta|}c(\theta)$. Observing that the relations \eqref{PropertiesFredholm} still hold, one obtains a one-parameter family $\{\mathrm{Ch}_t\}$ of Chern characters. Plugging the product formula \eqref{ProductOfCochains} into the perturbation series \eqref{PerturbationSeries}, one obtains the explicit but somewhat cumbersome combinatorial formula
\begin{equation} \label{DefinitionChern}
\begin{aligned}
  &\mathrm{Ch}_t(\theta_1, \dots, \theta_N) = \\
  & \!\!\!\! \sum_{\substack{k=1 \\ 1 \leq i_1 < \dots < i_k \leq N}}^N \!\!\!\!\!\!\!\! t^{|\theta|-N+2k} \int_{\Delta_k} \Str\Bigl( e^{-t^2 \tau_1 D^2} \prod_{p=1}^k F(\theta_{i_{p-1}+1}, \dots, \theta_{i_p}) e^{-t^2(\tau_{p+1}-\tau_{p})D^2}\Bigr) d \tau,
\end{aligned}
\end{equation}
for homogeneous elements $\theta_1, \dots, \theta_N \in \Omega_{\T }(M)$, where and $|\theta| = |\theta_0| + \dots + |\theta_N|$ is the total degree.
This formula can be understood as a certain time-ordered expectation value, where, since the operators $F$ vanish when more than two entries are filled, only neighboring $\theta_i$ ``interact''. 

\subsection{The localization formula}

We now aim to prove the localization formula \eqref{LocalizationFormula} for closed entire cyclic chains $c$, where the \emph{restriction map} is the map $i: \mathsf{B}(\Omega_{\T}(M)) \rightarrow \Omega(M)$ given by
\begin{equation} \label{RestrictionMap}
  i(\theta_1, \dots, \theta_N) = \frac{1}{N!}\theta^{\prime\prime}_1 \wedge \cdots \wedge \theta^{\prime\prime}_N;
\end{equation}
here, as always, $\theta_i = \theta_i^\prime + \sigma \theta_i^{\prime\prime} \in \Omega_{\T}(M)$. 

\begin{remark}
The above map satisfies $i(\theta_1, \dots, \theta_N) = j^*\rho(\theta_1, \dots, \theta_N)$, where $\rho$ is the iterated integral map \eqref{IteratedIntegralFormula} and $j^*$ denotes the pullback with respect to the inclusion $j: M \rightarrow \mathsf{L}M$ as the subset of constant loops (notice here that since the image of $j$ is the set of constant loops, the integral over $\Delta_N$ in \eqref{IteratedIntegralFormula} is constant and integration contributes a factor of $\mathrm{vol}(\Delta_N) = 1/ N!$). In this sense, $i$ is a combinatorial version of the pullback map $j^*$.
\end{remark}

Let $M$ be a compact Riemannian spin manifold of even dimension $d$, so that the family $\{(H, Q_t, c_t)\}_{t > 0}$ of Fredholm modules together with the corresponding family of Chern characters $\{\mathrm{Ch}_t\}_{t > 0}$ introduced in \S\ref{SectionChernCharacters} is defined. By Homotopy invariance of the Chern character \cite[Thm.~6.2]{BatuMatthias19}, for any $s, t >0$, there exists an analytic bar cochain $\mathrm{CS}_{s, t}$ such that, when restricted to $\mathsf{B}^\natural_\epsilon(\Omega_{\T}(M))$,
\begin{equation}
  \mathrm{Ch}_s - \mathrm{Ch}_t = \beta \mathrm{CS}_{s, t}; 
\end{equation}
in other words $\mathrm{Ch}_s$, $\mathrm{Ch}_t$ are cohomologous as cyclic chains. Explicitly, this means that for all entire cyclic chains $c \in \mathsf{B}^\natural_\epsilon(\Omega_{\T}(M))$, we have
\begin{equation*}
  \mathrm{Ch}_D(c) - \mathrm{Ch}_t(c) = \beta\mathrm{CS}_{1, t}(c) = - \mathrm{CS}_{1, t}(b(c)).
\end{equation*}
Therefore,  if $c$ is additionally \emph{closed}, i.e.\ $b(c)=0$, then $\mathrm{Ch}(c) = \mathrm{Ch}_t(c)$ for all $t > 0$. 

The above discussion shows that we can compute the value of $\mathrm{Ch}(c)$ by taking the limit of $\mathrm{Ch}_t(c)$, as $t \rightarrow 0$. The localization formula \eqref{LocalizationFormula} therefore follows from the following theorem, which is proved in \S\ref{SectionTangentGroupoid}.

\begin{theorem} \label{LimitTheorem}
  For all $\theta_1, \dots, \theta_N \in \Omega_{\T}(M)$, we have
  \begin{equation} \label{ChernLimitFormula}
   \lim_{t \rightarrow 0} \mathrm{Ch}_t(\theta_1, \dots, \theta_N) = \frac{1}{(2 \pi i)^{d/2}N!} \int_M \hat{A}(M) \wedge \theta_1^{\prime\prime} \wedge \cdots \wedge \theta_N^{\prime\prime},
  \end{equation}
  where the characteristic form $\hat{A}(M)$ is given by
\begin{equation} \label{AHatForm}
\hat{A}(M) = \det\,^{\!\!1/2}\left( \frac{ R/2}{\sinh( R/2)} \right).
\end{equation}
\end{theorem}

We finish this section with an application of the localization formula, which features the Bismut-Chern characters from Example~\ref{ExampleBismutChernCharacter}; compare \cite[\S8]{BatuMatthias19}.

\begin{proposition} \label{PropMcKean}
Let $p \in \Omega(M) \otimes M_n(\C)$ with $p^2 = p$ and define 
\begin{equation*} 
D_p = p D p + (1-p)D (1-p),
\end{equation*}
where by abuse of notation, $D$ denotes the Dirac operator on $S \otimes \underline{\C}^n$. Then
\begin{equation*}
\mathrm{Ch}_{\mathsf{D}}(\mathrm{Ch}(p)) = \Str(p e^{-D_p}).
\end{equation*}
\end{proposition}

By the McKean-Singer formula, the supertrace $\Str(p e^{-D_p})$ is just the index of the Dirac operator twisted with the vector bundle $E = \mathrm{im}(p)$. Since $i(\mathrm{Ch}(p))$ is easily worked out to be the Chern character of the bundle $E$, this combines with the localization formula \eqref{LocalizationFormula} to yield the Atiyah-Singer index theorem.

We remark that in the above argument, we have cheated a little bit, since $\mathrm{Ch}(p)$ is not closed strictly, but only closed modulo a certain subcomplex. However, since $\mathrm{Ch}_t$ and $\mathrm{CS}_{t, s}$ vanish on this subcomplex (see Thm.~5.5 of \cite{BatuMatthias19}), the localization formula also holds for the Bismut-Chern-characters.

\begin{proof}[Proof of Prop.~\ref{PropMcKean}]
Observe that $D_p = D + c((2p-1)dp)$ and with $\mathcal{R}$ defined as in \eqref{DefinitionR}, 
\[
\begin{aligned}
  F(\mathcal{R}) &= c((dp)^2) - [D, c((2p-1)dp)]\\
  F(\mathcal{R}, \mathcal{R}) &= c((2p-1)dp)^2 + c((dp)^2).
\end{aligned}
\]
Put together, 
\begin{equation*}
  D_p^2 = D^2 + [D, c((2p-1)dp)] + c((2p-1)dp)^2 = D^2 - F(\mathcal{R}) + F(\mathcal{R}, \mathcal{R}).
\end{equation*}
Writing $e^{-D_p}$ as a perturbation series, we therefore obtain
\begin{equation*}
  e^{-D_p^2} = \sum_{N=0}^\infty (-1)^N \int_{\Delta_N} e^{-\tau_1 D^2} \prod_{k=1}^N(F(\mathcal{R}, \mathcal{R}) - F(\mathcal{R})) e^{-(\tau_{k+1}-\tau_k)D^2} d\tau
\end{equation*}
By the cyclic permutation property of the supertrace, multiplying this by $p$ and taking the supertrace yields $\mathrm{Ch}_D(\mathrm{Ch}(p))$.
\end{proof}

\section{The tangent groupoid and the localization formula} \label{SectionTangentGroupoid}

In this section, we first give a brief introduction to the tangent groupoid and the rescaled spinor bundle and then use the techniques introduced to prove Thm.~\ref{LimitTheorem}.

\subsection{The scaling order} \label{SectionScalingOrder}

Let $M$ be an even-dimensional spin manifold with spinor bundle $S$.  In this section, we briefly review the notion of scaling order for sections of the bundle $S \boxtimes S^*$ over $M \times M$, the fiber of which over $(m_1, m_2)$ is $S_{m_1} \otimes S_{m_2}^*$. For a more detailed treatment, we refer to \cite[\S3.3]{HigsonYi19}.

To begin with, denote by $\Cliff(T_mM)$ the Clifford algebra of $T_m M$ and let
\begin{equation*}
  c: \Lambda^* T_mM \longrightarrow \Cliff(T_mM)
\end{equation*}
be the \emph{quantization map}, defined by $e_{i_1} \wedge \cdots \wedge e_{i_k} \mapsto e_{i_1} \cdots e_{i_k}$ in terms of an orthonormal basis $e_1, \dots, e_n \in T_m M$; see \cite[\S3.1]{BerlineGetzlerVergne92} for details. $c$ is not an algebra homomorphism, but defining $\Cliff_k(T_mM)$ to be the image of $\Lambda^{\leq k} T_mM$ under $c$ defines a filtration on the Clifford algebra. An element $a$ of the Clifford algebra is said to have \emph{Clifford order} $k$ or less if it is contained in $\Cliff_k(T_mM)$.
For $a \in \Cliff(T_mM)$, we denote by $[a] \in \Lambda^*T_mM$ its inverse image under the quantization map (often called the \emph{Clifford symbol}) and if $a \in \Cliff_k(T_mM)$, we let $[a]_k$ be the $k$-form component of $[a] \in \Lambda^{\leq k} T_m M$.

A differential operator $P$ has \emph{Getzler order} $p$ or less if locally, it can be written as
\begin{equation*}
  P = f D_1 \cdots D_p,
\end{equation*}
where $f$ is a smooth function, and each $D_i$ is either a covariant derivative $\nabla_X$, a Clifford multiplication $c(X)$, or the identity operator. The definition of scaling order now uses the fact that on the diagonal of $M \times M$, we have the identification
\begin{equation*}
  (S \boxtimes S^*)_{(m, m)} \cong S_m \otimes S_m^* \cong \End(S_m) \cong \Cliff(T_mM).
\end{equation*}

\begin{definition} \cite[\S3.4]{HigsonYi19} $-$ Let $p \in \Z$. We say that a section $s$ of $S \boxtimes S^*$ has \emph{scaling order $p$ or more} if for every $m \in M$, 
\begin{equation*}
  \CliffordOrder\bigl(D s(-, m)|_m\bigr) \leq q - p
\end{equation*}
for every differential operator $D$ of Getzler order $q$ or less, acting on the first component of $s$.

\end{definition}

\subsection{The tangent groupoid and the rescaled spinor bundle}
	The  tangent groupoid was introduced by Alain Connes to give a simple and elegant proof of the Atiyah-Singer index theorem \cite[Chapter~2, Section~5]{Connes94}. Given smooth manifold $M$, the tangent groupoid $\T M$ is 
	a smooth manifold whose underlying set is
	\[
	\T M = (TM \times \{0\}) \sqcup (M\times M\times \mathbb{R}^\times).
	\] 
	If $M\supset U \xrightarrow{\varphi} \mathbb{R}^n$ is a local coordinate chart, then $\T U \subset \T M$ is an open subset and there is a local coordinate chart
	\[
	\T U \xrightarrow{\phi} \mathbb{R}^{2n+1}
	\]
	given by
	\begin{equation}\label{eq-manifold-structure}
	\begin{cases}
	(x, m ,t)\to (\frac{\varphi(x)-\varphi(m)}{t},\varphi(m),t) \\
	(X,m,0) \to (\varphi_\ast X, \varphi(m),0).
	\end{cases}
	\end{equation}
	
	In \cite{HajSaeediSadeghHigson18}, the authors adopt a more algebraic way towards the tangent groupoid, namely it is built as spectrum of the following algebra.

	\begin{definition} 
Denote by 
$A(\T M)\subseteq C^\infty(M\times M)[t^{-1},t]$    the $\R$-algebra of those Laurent polynomials
\begin{equation}\label{eq-define-laurent}
\sum_{p\in \mathbb{Z}} f_p t^{-p}
\end{equation}
for which  each coefficient $f_p$  is a  smooth, real-valued  function on $M\times M$ that vanishes to order $\geq p$ on $M$ (and all but finitely many $f_p$ are zero).
\end{definition}

In general, the spectrum of an algebra comes naturally with a 
topology, the Zariski topology.
In this particular case, the spectrum of $A(\T M)$ turns out to have a smooth manifold structure which coincides with the manifold structure on $\T M$ defined above. A Laurent polynomial of the form \eqref{eq-define-laurent} naturally defines a smooth function on the tangent groupoid $\T M$, and the evaluation maps are given by
\[
\begin{aligned}
\varepsilon_{(x,m,\lambda)}: \sum_{p\in \mathbb{Z}} f_p t^{-p}  &\longmapsto \sum_{p\in \mathbb{Z}} f_p(x,m) \lambda^{-p}\\
\varepsilon_{(X,m)}: \sum_{p\in \mathbb{Z}} f_p t^{-p} &\longmapsto \sum_{p\in \mathbb{Z}} \frac{1}{p!}X^p(f_p).
\end{aligned}
\]
The set of smooth functions on $\T M$ is locally smoothly generated by these functions (see \cite[Lemma~2.4]{HajSaeediSadeghHigson18}).

\medskip

Let $M$ be an even dimensional spinor manifold with spinor bundle $S\to M$. In order to introduce Getzler's rescaling technique in the context of the tangent groupoid, by deforming  $S$, we build a vector bundle $\mathbb{S} \to \T M$ over the tangent groupoid, following the construction in  \cite{HigsonYi19}.  This bundle is called the rescaled bundle and it is built from the following $A(M)$-module.

 \begin{definition} \label{DefinitionSTM}
Denote by  $S(\T M)$   the complex  vector space of   Laurent polynomials
\begin{equation}\label{eq-def-laurent-sec}
\sum_{p \in \mathbb{Z}} s_p  t^{-p}
\end{equation}
where each $s_p$ is a smooth section of $S\boxtimes S^*$ of scaling order at least $p$.  
\end{definition}


The complex vector space $S(\T M)$ so constructed is indeed a $A(M)$-module, the module structure is given by the Laurent polynomial multiplication. It turns out the module $S(\T M)$ can be made into a sheaf of locally free modules over the sheaf of smooth functions on $\mathbb{T} M$, thus giving rise to the rescaled bundle $\mathbb{S} \to \mathbb{T}M$.  

A Laurent polynomial of the form \eqref{eq-def-laurent-sec} naturally defines a smooth section of $\mathbb{S}$ whose evaluation map is given by
\begin{align}\label{eq-evaluation-interior}
\varepsilon_{(x,m,\lambda)}: \sum_{p \in \mathbb{Z}} s _p  t^{-p} &\longmapsto \sum_{p \in \mathbb{Z}} s _p(x,m)  \lambda^{-p} \\
\label{eq-evaluation-boundary}
\varepsilon_{(X,m)}: \sum_{p \in \mathbb{Z}} s _p  t^{-p} &\longmapsto \sum_{q,p}\frac{1}{q!}\bigl[\nabla_X^q s_p(-,m)|_m \bigr]_{q-p} 
\end{align}
where $\nabla_X$ is the covariant derivative acting on the first variable of $S\boxtimes S^\ast$ and $[\cdot]_k$. Observe here that since $s_p$ has scaling order at least $p$, $\nabla_X^q s_p(-,m)|_m \in \Cliff(T_mM)$ has Clifford order at most $q-p$ at each $m \in M$, hence its $(q-p)$-th Clifford symbol is well-defined. 
A general smooth section $f$ of $\mathbb{S}$ can locally be written as a finite sum
\begin{equation}\label{eq-local-form}
f = \sum_j f_j\cdot s_j
\end{equation}
where $f_j \in C^\infty(\mathbb{T} M)$ and the $s_j$ are Laurent polynomials of the form \eqref{eq-def-laurent-sec}, which determine smooth sections of $\mathbb{S}$ denoted by the same symbol.

Set theoretically, over $M\times M\times \{t\}$, the rescaled bundle is the tensor product bundle $S\boxtimes S^\ast$ and the fibers over $TM\times \{0\}$ is the pullback of the exterior bundle $\wedge^\ast T^\ast M \to M$ along $\pi: TM\to M$. 
	\begin{align*}
	\xymatrix{
	\mathbb{S} \ar[d]  \\
	\T M 
	}
	&
	\,
	&
	\xymatrix @C=0.1pc {
	\pi^\ast\wedge^\ast T^\ast M \ar[d] & &S\boxtimes S^\ast \ar[d]\\ 
	TM\times \{0\} &\sqcup &M\times M\times \mathbb{R}^\times
	}
	\end{align*}
The space of compactly supported smooth sections of the rescaled bundle has an algebra structure in the following way: For $f,g\in C^\infty_c(\T M,\mathbb{S})$, $f\ast g \in C^\infty_c(\T M, \mathbb{S})$ is defined by
\begin{equation}\label{twisted-convolution}
\begin{split}
(f\ast g)(x,m,t) &= \int_M f(x,y,t)g(y,m,t)t^{-n}dy \quad \\
(f\ast g)(X,m,0) &= \int_{T_mM} f(X-Y,m,0)g(Y,m,0)e^{-\frac{1}{2}[\kappa(X,Y)]}dY
\end{split}
\end{equation}
where $(x,m,t)\in M\times M\times \mathbb{R}^\times$ and $(X,m,0)\in TM\times \{0\}$ and where $\kappa$ is the curvature tensor of the spinor bundle (so that $\kappa(X, Y) \in \Cliff(T_mM)$ for $X, Y \in T_mM$) and $[\kappa(X,Y)] \in \Lambda T_m M$ is the inverse image of the Clifford algebra element $\kappa(X,Y)$ under the quantization map  (see \cite[Section~5.2]{HigsonYi19}). Crucially, we will use the following result.

\begin{theorem}{\normalfont \cite[Theorem~5.4.2]{HigsonYi19}} $-$  \label{ThmContinuityTrace}
For each $t \in \R$, the formula
\begin{equation}\label{eq-smooth-supertrace}
\begin{split}
\Str_t(f) &= \int_M \str(f(m,m,t))t^{-n} dm \qquad \text{ for } t\neq 0, \\
\Str_0(f) &= \left(\frac{2}{i}\right)^{d/2}\int_{M} f(0,-,0),
\end{split}
\end{equation}
defines a supertrace on $C^\infty_c(\T M,\mathbb{S})$; here in the second integral, $f(0,-,0)$ is a differential form on $M$, which is integrated using the orientation on $M$. Moreover, the map $t \mapsto \Str_t(f)$ is smooth.
\end{theorem}

\begin{remark}
For $t\neq 0$, the traces $\Str_t$ can be viewed as an integral over the $t$-fiber of $\mathbb{T}M$, when the fibers are equipped with the rescaled metric $t^{-2} g$.
The theorem then asserts that the formula for $\Str_0$ is the continuous (even smooth) extension of this family to the fiber over $t=0$.
\end{remark}

\subsection{Rapidly decaying sections}

Let $M$ be a compact Riemannian spin manifold of even dimension $d$. A disadvantage of the algebra $C^\infty_c(\T M,\mathbb{S})$ considered above is that it is too small to contain the ``heat kernel element'' $e^{-t^2D^2}$. In this section, we shall construct an enlargement $\mathcal{S}(\T M,\mathbb{S})$ of this algebra consisting of sections of rapid decay, in particular $e^{-t^2D^2}$, and this still supports the family of supertraces \eqref{eq-smooth-supertrace}.

\begin{definition}
\label{def-first-seminorm}
Let $f$ be a compactly supported smooth section of the rescaled bundle. Define a family of norms $\{N_{k}\}$, $k \in \N$, on $C^\infty_c(\T M,\mathbb{S})$ by
\begin{align}
N_{k} (f)&=\sup_{(x,m,t)\in M\times M\times \mathbb{R}^\times} \left(1+\frac{d(x, m)^2}{t^2}\right)^{k/2} \Big| f(x,m,t)\Big| \label{first-seminorm}  
\end{align}
where  $d(x, m)$ is the Riemannian distance between $x$ and $m$ and let
	\[
	\mathcal{S}(\T M,\mathbb{S}) := \Big\{f\in C(\T M,\mathbb{S}) \mid \forall k \in \N: N_k(f)<\infty \Big\}.
	\]
\end{definition}

\begin{lemma} \label{LemmaExtensionRapidDecay}
The following holds:
\begin{enumerate}[{\normalfont (1)}]
\item $\mathcal{S}(\T M,\mathbb{S})$ is complete and $C^\infty_c(\T M, \mathbb{S}) \subset \mathcal{S}(\T M,\mathbb{S})$ is dense.
\item The convolution product extends continuously to a product on $\mathcal{S}(\T M,\mathbb{S})$.
\item Each of the supertraces \eqref{eq-smooth-supertrace} extends continuously to $\mathcal{S}(\T M,\mathbb{S})$ and for each $f \in \mathcal{S}(\T M,\mathbb{S})$, $\mathrm{Str}_t(f)$ is continuous in $t$.
\end{enumerate}
\end{lemma}

\begin{proof}
(1) Let $\{\varphi_\alpha: M\supset U_\alpha \rightarrow \mathbb{R}^n\}$ be a collection of coordinate charts on $M$ and let $\{\phi_\alpha: \T U_\alpha \rightarrow \mathbb{R}^{2n+1}\}$ be the induced coordinate charts on $\T M$, as given in \eqref{eq-manifold-structure}. One then easily shows that restricted to $\T U_\alpha$, the seminorm $N_k$ is equivalent to the seminorm
	\[
	\sup_{(a,b,t)\in \phi_\alpha(\T U_\alpha)} (1+|a|^2)^{k/2} |f \circ \phi^{-1}_\alpha (a,b,t)|.
	\]
The rest follows from routine arguments.

\noindent (2) We calculate
\begin{align*}
N_{k}(f\ast g) &=\sup_{(x,m,t)} \left(1+\frac{d(x, m)^2}{t^2}\right)^{k/2} \Big|\int_M  f(x,y,t)g(y,m,t) t^{-n}dy\Big|  \\
&\leq C_k \sup_{(x, m, t)}  \Big| \int_M \left(1+\frac{d(x, y)^2}{t^2}\right)^{k/2}   f(x,y,t) g(y,m,t) t^{-n}dy\Big| \\
&\quad +C_k \sup_{(x, m, t)}  \Big| \int_M   f(x,y,t) \left(1+\frac{d(y, m)^2}{t^2}\right)^{k/2} g(y,m,t) t^{-n}dy\Big| \\
&\leq C_k N_{k+n+1}(f) N_{0}(g) \sup_{(x, t)}\int_M  \left(1+\frac{d(x, y)^2}{t^2}\right)^{-(n+1)/2}t^{-n}dy \\
&\quad+ C_k N_{n+1}(f) N_{k}(g) \sup_{(x, t)}\int_M  \left(1+\frac{d(x, y)^2}{t^2}\right)^{-(n+1)/2}t^{-n}dy
\end{align*}
where $C_k$ is a constant such that for all $a, b \geq 0$, 
\[
(1+a+b)^{k/2}\leq C_k(1+a)^{k/2}+C_k(1+b)^{k/2}.
\]
It remains to show that the integral 
\begin{equation}\label{eq-integral-distance}
\int_M\left(1+\frac{d(x, y)^2}{t^2}\right)^{-(n+1)/2}t^{-n}dy 
\end{equation}
is uniformly bounded with respect to $t\in \mathbb{R}^\times$. For $\varepsilon >0$, split the integral up in one integral over $M \setminus B_\varepsilon(x)$ and one over $B_\varepsilon(x)$. The first of these is clearly bounded, and the second one can be estimated by the integral
\begin{equation*}
  \int_{B_\varepsilon(0)} \left(1+\frac{|v|^2}{t^2}\right)^{-(n+1)/2}t^{-n}d v =  \int_{B_{\varepsilon/t}(0)} \left(1+|\xi|^2\right)^{-(n+1)/2} d \xi
\end{equation*}
over Euclidean space. In the second step we replaced $\xi = v/t$ to obtain an expression which is clearly bounded.

\noindent (3) For $t \neq 0$, the formula \eqref{eq-smooth-supertrace} clearly extends to a continuous linear functional on $\mathcal{S}(\T M, \mathbb{S})$. In the case $t=0$, we observe that by (1) above, elements $f \in \mathcal{S}(\T M, \mathbb{S})$ satisfy $|f(X, m, 0)| \leq C_k (1 + |X|^2)^{k/2}$ for any $k \in \N$. Hence also the second formula in \eqref{eq-smooth-supertrace} extends to a continuous linear functional on $\mathcal{S}(\T M, \mathbb{S})$. To show continuity at $t=0$, let $f_i \in C_c(\T M, \mathbb{S})$ be compactly support sections with $f_i \to f$ in $\mathcal{S}(\T M,\mathbb{S})$.
Then
\[
 |\Str_t(f)- \Str_0(f) | \leq |\Str_t(f) - \Str_t(f_i)| + |\Str_t(f_i) - \Str_0(f_i)| + |\Str_0(f_i) - \Str_0(f)|.
\]
The second term converges to zero by Thm.~\ref{ThmContinuityTrace}, and the first and the third by continuity of $\Str_t$. This finishes the proof.
\end{proof}

\subsection{The heat kernel element}

We now show that the space $\mathcal{S}(\T M, \mathbb{S})$ of rapidly decaying sections of $\mathbb{S}$ contains the ``heat kernel element'' $e^{-t^2D^2}$.

\begin{lemma}\label{lem-zero-boundary-value}
Let $f$ be a smooth section of $S\boxtimes S^\ast \to M\times M\times \mathbb{R}$. Then $t^{n+1}f$ defines a smooth section of the rescaled bundle $\mathbb{S}\to \T M$ such that 
\[
(t^{n+1}f)(\gamma)=
\begin{cases}
t^{n+1}f(x,y,t) & \gamma=(x,y,t)\\
0 & \gamma=(X,m,0).
\end{cases}
\]
\end{lemma}

\begin{proof}
As a $C^\infty(M\times M\times \mathbb{R})$ module, $C^\infty(M\times M\times \mathbb{R}, S\boxtimes S^\ast)$ is locally finitely generated and free. We could choose $s_1,s_2,\cdots, s_p$ a sequence of local sections of $S\boxtimes S^\ast \to M\times M$ such that $f$ can be locally written as combination of $s_1,\cdots, s_p$. That is
\[
t^{n+1}f(x,y,t) = f_1(x,y,t) t^{n+1}s_1(x,y)+  \cdots + f_p(x,y,t)t^{n+1} s_p(x,y)
\]
locally, for some smooth functions $f_1,f_2,\cdots, f_p$ on $M\times M\times \mathbb{R}$.
Here $t^{n+1}s_i$ defines local section of the rescaled bundle and its value at $(X,m,0)$ which can be evaluated by \eqref{eq-evaluation-boundary} is clearly zero.
\end{proof}

\begin{proposition} \label{PropHeatKernel}
For each $\tau >0$, there is $H_\tau \in \mathcal{S}(\T M,\mathbb{S})$ such that for $t >0$, 
\begin{equation}\label{heat-element-1}
H_\tau (x,m,t) =t^ne^{-t^2\tau D^2}(x,m),
\end{equation}
where $e^{-t^2\tau D^2}(x,m)$ is the heat kernel of $D$.
Moreover, this element satisfies
\begin{equation}\label{heat-element-2}
H_\tau(X,m,0) = \frac{1}{(4 \pi \tau)^{n/2}} \det\,^{\!\!1/2}\left( \frac{  \tau R/2}{\sinh( \tau R/2)} \right) \exp \left( - \frac{1}{4 \tau} \left\langle X, \frac{\tau R}{2} \coth\left(\frac{\tau R}{2}\right) X \right\rangle \right),
\end{equation}
where $R$ is the Riemannian curvature tensor, interpreted as a skew-adjoint matrix of differential 2-forms in a local frame.
\end{proposition}

\begin{proof}
We need a slight refinement of Thm.~4.1 of \cite{BerlineGetzlerVergne92}.
By the asymptotic expansion of the heat kernel, near  the diagonal in $M \times M$, we have
\begin{equation}\label{eq-asymptotic-expansion-heat-kernel}
\begin{split}
t^d\cdot e^{-t^2\tau\mathsf{D}^2}(x,m) = \frac{1}{(4\pi \tau)^{d/2}}e^{\frac{-d(x, m)^2}{4t^2\tau}} 
\sum_{i=0}^{d/2} t^{2i} \tau^i \Theta_i(x, m) 
+ \mathcal{O}(t^{n+1})
\end{split}
\end{equation}
where the $\Theta_i(-, m)$ are determined by system of differential equations
\[
\begin{cases}
\nabla_\mathcal{R} \Theta_0(-, m) = 0 \\
\left( \nabla_\mathcal{R}+i \right) \Theta_i(-, m)= -B \Theta_{i-1}(-, m)
\end{cases}
\]
with initial condition $\Theta_0(m,m) = 1$, where $\mathcal{R}$ is the radial vector field associated with a Riemannian normal coordinate system around $m$, $B$ is a differential operators on $S$ of Getzler order $2$ (see \cite[\S2.5]{BerlineGetzlerVergne92} for details). We claim that $\Theta_i$ has scaling order $2i$. The first equation $\nabla_\mathcal{R} \Theta_0 = 0$ says that $\Theta_0(x,m) = P(x,m)$, the parallel translation operator which has scaling order $0$ according to \cite[Prop.~3.3.10]{HigsonYi19}. The rest can be shown by an induction argument: Assume that $\Theta_{i-1}$ has scaling order $2i-2$; then $B\Theta_{i-1}$ has scaling order $2i$. Since $\mathcal{R}$ vanishes at $m$, $\nabla_\mathcal{R}$ does not change the scaling order so that by the differential equation, $\Theta_i$ also has scaling order $2i$. According to \eqref{eq-local-form}, the sum of the first $n$ terms in the asymptotic expansion defines a smooth section of $\mathbb{S}$. The remainder term $\mathcal{O}(t^{n+1})$ is a section of $S\boxtimes S^\ast \to M\times M\times \mathbb{R}$ which satisfies the condition of Lemma~\ref{lem-zero-boundary-value}. Overall, $H_\tau$ defines an element in $C^\infty(\mathbb{T}M,\mathbb{S})$. 

Next we shall show $N_k(H_\tau)<\infty$ for all $k$. If $x\neq y$, it is well known that the heat kernel rapidly decreasing as $t\to 0$. We only have to consider the case when $(x,y)$ is very close to the diagonal. In that case, the estimate is done by using the asymptotic expansion \eqref{eq-asymptotic-expansion-heat-kernel}. Indeed, $\Theta_i(x,y)$ are all bounded near the diagonal, and 
\[
\left(1+\frac{d(x, m)^2}{t^2}\right)^{k/2} e^{\frac{-d(x, m)^2}{4t^2\tau}}
\]
is uniformly bounded in $(x, m, t)$ for any given $k$ and $\tau > 0$. Therefore, $H_\tau \in \mathcal{S}(\mathbb{T}M,\mathbb{S})$.

The value of $H_\tau(X,m,0)$ is calculated for example in \cite[Thm~4.20]{BerlineGetzlerVergne92} or \cite[Prop.~12.25]{RoeBook} with other means. However, \eqref{heat-element-2} can also be obtained within the framework of \cite{HigsonYi19}, as we explain now. Because $D^2$ has Getzler order two, the results of \S3.6 in \cite{HigsonYi19} imply that $t^2 D^2$ extends to an operator $\boldsymbol{D}^2$ on $\mathbb{T}M$, acting on sections of $\mathbb{S}$; over $t=0$, it is given by is Getzler symbol, as computed e.g.\ in \cite[Prop.~12.17]{RoeBook}. For $X \in T_mM \subset \mathbb{T}M$, the formula is
\begin{equation} \label{owpeif}
\varepsilon_{X}(\mathbf{D}^2 s) = L . \varepsilon_{X}(s) \qquad \text{with} \qquad L = \sum_{i=1}^d \left(\frac{\partial}{\partial X_i} - \frac{1}{4}\sum_{j=1}^d R_{ij}X_j \right)^2.
\end{equation}
Here $\varepsilon_{X} = \varepsilon_{(X, m)}$ is the point evaluation map \eqref{eq-evaluation-boundary} and the components $X_i$ of $X$ and the $R_{ij} \in \Lambda^2T_mM$ are the components of the curvature tensor of $M$ defined with respect to some orthonormal basis of $T_m M$.  

We want to show that for any $\tau > 0$, we have
\begin{equation} \label{eq-symbol-d-square}
\varepsilon_{X}\left(\exp(-\tau\boldsymbol{D}^2)s\right)= \exp(-\tau L). \varepsilon_{X}(s). 
\end{equation}
It suffices to verify this for all $s$ in the $\mathcal{A}(M)$-module $S(\mathbb{T}M)$ (remember Def.~\ref{DefinitionSTM}), as the general section is a linear combination over $C^\infty (\mathbb{T}M)$ of elements of $S(\mathbb{T}M)$ and one easily checks that the formula \eqref{eq-symbol-d-square} still holds when replacing $s$ by $f \cdot s$ for $f \in C^\infty (\mathbb{T}M)$. On $S(\mathbb{T}M)$,  $\boldsymbol{D}^2$ acts as
\[
\boldsymbol{D}^2: \sum_p s_p t^{-p} \mapsto \sum_p D^2(s_p) t^{-p+2}.
\]
Let $S_0(\mathbb{T}M)$ be the quotient of $S(\mathbb{T}M)$ by the subspace $t\cdot S(\mathbb{T}M)$. Since the point evaluations $\varepsilon_{X}$ are zero on $t\cdot S(\mathbb{T}M)$, they descend to $S_0(\mathbb{T}M)$ and it suffices to verify \eqref{eq-symbol-d-square} for $s \in S_0(\mathbb{T}M)$. However, since any section of $S \boxtimes S^*$ has scaling order at least $-d$, we see that the action of $(\tau\boldsymbol{D}^2)^N$ is zero on $S_0(\mathbb{T}M)$ for $N$ sufficiently large. Hence both sides of \eqref{eq-symbol-d-square} are actually given by an exponential series truncated at some finite $N$, so that \eqref{eq-symbol-d-square} follows from \eqref{owpeif}.

On the other hand, by \eqref{twisted-convolution},
\begin{equation}\label{eq-twisted-convolution-d-square}
\varepsilon_X\left(\exp(-\tau\boldsymbol{D}^2)s\right) = \int_{T_mM} \varepsilon_{X-Y}(H_\tau)e^{-\frac{1}{2}[\kappa(X, Y)]} \varepsilon_Y(s)dY.
\end{equation}
Equations \eqref{eq-symbol-d-square} and \eqref{eq-twisted-convolution-d-square} together imply 
\begin{equation}\label{eq-mehler-kernel-exponential}
\exp(-\tau L)(X,Y) = \varepsilon_{X-Y}(H_\tau)e^{-\frac{1}{2}\kappa(X,Y)}
\end{equation}
in particular, $\exp(-\tau L)(X,0) = \varepsilon_X(H_\tau)$ which combined with Mehler's formula (see e.g.\ \cite[\S4.2]{BerlineGetzlerVergne92}) verifies \eqref{heat-element-2}.
\end{proof}

\begin{remark} \label{RemarkTwistedConvIdentity}
Fix $m \in M$. The full integral kernel $\tilde{H}_\tau(X, Y) := e^{-\tau L}(X,Y)$ of the heat operator $e^{-\tau L}$ on $T_m M$ is given by \emph{Mehler's formula},
	\begin{align*}
&\tilde{H}_\tau(X,Y) = (4\pi)^{-n/2}\cdot\det\left(\frac{\tau R/2}{\sinh(\tau R/2)}\right)^{1/2} \times \\ 
&~\times \exp\left(-\langle X, \frac{ R}{8}\coth\left(\frac{\tau R}{2}\right)X\rangle +  \langle X,e^{\tau R/2}\frac{ R}{4}\cosech\left(\frac{\tau R}{2}\right)Y \rangle -\langle Y, \frac{ R}{8}\coth\left(\frac{\tau R}{2}\right)Y \rangle \right),
	\end{align*}
see \cite[\S4.2]{BerlineGetzlerVergne92}, which satisfies the convolution identity
\begin{equation*}
  \tilde{H}_{\tau+\tau^\prime}(X, Z) = \int_{T_mM} \tilde{H}_{\tau}(X, Y) \tilde{H}_{\tau^\prime}(Y, Z) d Y.
\end{equation*}
Now one can check that $\tilde{H}_\tau(X, Y) = H_\tau(X-Y, m, 0)e^{-\frac{1}{2}[\kappa(X, Y)]}$, hence the element $H_\tau$ from above satisfies the \emph{twisted} convolution identity
\begin{equation*}
  H_{\tau+\tau^\prime}(X, m, 0) = (H_{\tau} * H_{\tau^\prime})(X, m, 0) = \int_{T_mM} H_\tau(X-Y, m, 0) H_{\tau^\prime}(Y, m, 0) e^{-\frac{1}{2}[\kappa(X, Y)]} d Y.
\end{equation*}
Of course, this twisted convolution identity $H_{\tau+\tau^\prime} = H_\tau \ast H_{\tau^\prime}$ also follows from the semigroup property of $e^{-t^2 \tau D^2}$, which holds for $t \neq 0$ and by continuity must continue to hold at zero. However, the above calculations show that the factor of $e^{-\frac{1}{2}[\kappa(X, Y)]}$ appearing in the formula \eqref{twisted-convolution} for the twisted convolution precisely accounts for the failure of the Mehler kernel to be translation invariant.
\end{remark}

\subsection{Proof of Theorem \ref{LimitTheorem} }

Let $M$ be a compact Riemannian spin manifold of even dimension $d$ and $t >0$. Given $\theta_1, \dots, \theta_N$, the explicit formula \eqref{DefinitionChern} reveals that the corresponding Chern character $\mathrm{Ch}_t(\theta_1, \dots, \theta_N)$ is a sum of terms of the form
\begin{equation} \label{ExpressionToEvaluate}
t^{|\theta|+N - 2k}\int_{\Delta_k} \Str\Bigl(  e^{-t^2 \tau_1 D^2} \prod_{p=1}^k F(\theta_{i_{p-1}+1}, \dots, \theta_{i_p}) e^{-t^2(\tau_{p+1}-\tau_{p})D^2}\Bigr) d \tau,
\end{equation}
where $k \leq N$ and a $1 \leq i_1 < \dots i_k \leq N$ are given, and where $|\theta| = |\theta_0| + \dots + |\theta_N|$; we assume each $\theta_i$ to be homogeneous throughout. Recall moreover that
\begin{equation*}
\begin{split}
F(\theta) &= c(d\theta^\prime) - [D, c(\theta^\prime)] - c(\theta^{\prime\prime}), \\
F(\theta_1,\theta_{2}) &= (-1)^{|\theta_1|} \left(c(\theta_1^\prime)c(\theta_{2}^\prime)-c(\theta_1^\prime \wedge \theta_{2}^\prime)\right).
\end{split}
\end{equation*}
 The goal is now to calculate the limit as $t \rightarrow 0$ of the expression, which  will occupy the rest of the section.

\begin{lemma} \label{LemmaGetzlerSymbolsF}
Let $\theta, \theta_1, \theta_2 \in \Omega_{\T }(M)$ be homogeneous. Then each of the operators
\begin{equation*}
  t^{|\theta|+1} F(\theta), \qquad t^{|\theta_1|+|\theta_2|} F(\theta_1, \theta_2),
\end{equation*}
acting on sections of $S \boxtimes S^*$ over $M\times M \times \R^\times$ with respect to the first variable, extends smoothly to an operator acting on sections of $\mathbb{S}$ over $\T M$. Moreover, over $TM \times \{0\} \subset \T  M$, these extensions are given by 
\begin{equation*}
  \theta^{\prime\prime} \wedge f, \qquad \text{respectively} \qquad 0.
\end{equation*}
\end{lemma}

\begin{proof}
Each of the operators $F(\theta)$, $F(\theta_1, \theta_2)$ can (locally) be written as a composition of Clifford multiplication and covariant derivatives, therefore it follows from Lemmas~3.6.2 \& 3.6.3 of \cite{HigsonYi19} that when multiplied by $t^\ell$ for $\ell$ less than or equal to their Getzler order, they extend smoothly to all of $\T M$ and their action over the $t=0$ slice is given by their Getzler symbol. We deal with them in turn.

\noindent (a) Regarding the operator $F(\theta)$, suppose that $\theta = \theta^\prime + \sigma \theta^{\prime\prime}$ has total degree $|\theta| = \ell$, meaning that $\theta^\prime \in \Omega^\ell(M)$ and $\theta^{\prime\prime} \in \Omega^{\ell+1}(M)$. A local calculation shows that  in terms of a local orthonormal basis $e_1, \dots, e_n$, one has the formula
\begin{equation*}
  c(d\theta^\prime) - [D, c(\theta^\prime)] = - 2 \sum_{i=1}^n c(e_i \lrcorner \theta) \nabla_{e_i} - c(\delta \theta),
\end{equation*}
where $\lrcorner$ denotes insertion of vectors into differential forms and $\delta$ is the codifferential. Since for each $i$, $c(e_i \lrcorner \theta)$ can be written as a sum of composites of $\ell-1$ Clifford multiplications, the right hand has Getzler order at most $\ell$. We obtain that $c(d\theta^\prime) - [D, c(\theta^\prime)]$ is of lower order compared to $c(\theta^{\prime\prime})$, which has Getzler order $\ell+1$ (as $c(\theta^{\prime\prime})$ can be written as a sum of composites of $\ell+1$ Clifford multiplications). Hence $t^{\ell+1} F(\theta)$ extends smoothly to all of $\T M$, and over $t=0$, we have $t^{\ell+1} F(\theta) = t^{\ell+1} c(\theta^{\prime\prime})$. It then follows from Lemma~3.6.2 in \cite{HigsonYi19} that over $t=0$, $t^{\ell+1}c(\theta^{\prime\prime})$ is given by wedging with $\theta^{\prime\prime}$.

\noindent (b) Looking at the formula for $F(\theta_1, \theta_2)$, it is clear that it has Getzler order at most $\ell_1 + \ell_2$ (where $\ell_i = |\theta_i|$), hence $t^{\ell_1+\ell_2}$ extends continuously to all of $\T M$, and over $t=0$, it is given by
\begin{equation*}
  (-1)^{\ell_1}\Bigl( [c(\theta_1^\prime)]_{\ell_1} \wedge [c(\theta_2^\prime)]_{\ell_2} \wedge f - [c(\theta_1^\prime \wedge \theta_2^\prime)]_{\ell_1+\ell_2} \wedge f \Bigr) 
  = 0
\end{equation*}
where $[-]_{\ell_i}$ denotes the Clifford symbol (see \S\ref{SectionScalingOrder}).
\end{proof}

We are now in the position to prove the following result, which implies Thm.~\ref{LimitTheorem} and hence finishes the proof of the localization formula.

\begin{proposition}
  If $k < N$, the expression \eqref{ExpressionToEvaluate} converges to zero, as $t \rightarrow 0$. In the case $k=N$, the limit is given by the right hand side of \eqref{ChernLimitFormula}.
\end{proposition}

\begin{proof}
Observe that since $F$ whenever if more than two $\theta_i$ are inserted, the expression \eqref{ExpressionToEvaluate} can be non-zero only if $i_p - i_{p-1} \leq 2$ for all $p=1, \dots, k$; we assume throughout that this is the case. We set
\begin{equation*}
  \ell_p = \begin{cases}  |\theta_{i_p}|+1 & \text{if}~~ i_p - i_{p-1} = 1 \\
  |\theta_{i_p-1}| + |\theta_{i_p}| & \text{if} ~~ i_p - i_{p-1} = 2.
  \end{cases}
\end{equation*}
Observe that $\ell_p$ is precisely the Getzler order of $F(\theta_{i_{p-1}+1}, \dots, \theta_{i_p})$, as seen in the proof of Lemma~\ref{LemmaGetzlerSymbolsF}. Hence if we set
\begin{equation*}
  K^p_\tau = t^{n+\ell_p} F(\theta_{i_{p-1}+1}, \dots, \theta_{i_p}) e^{-t^2\tau D^2},
\end{equation*}
then by Lemma~\ref{LemmaGetzlerSymbolsF} and Prop.~\ref{PropHeatKernel}, each $K^p_\tau$ (a priori defined only over $M \times M \times \R^\times$) extends smoothly to a section of the bundle $\mathbb{S} \rightarrow \T  M$; in fact an element of $\mathcal{S}(\T M, \mathbb{S})$.

Because necessarily $i_p - i_{p-1} = 1$ for $2k-N$ many $p$ and  $i_p - i_{p-1} = 2$ for $N-k$ many $p$ ($p \geq 1$), we have
\begin{equation*}
  \ell_0 + \dots + \ell_k = |\theta| + 2k -N.
\end{equation*}
Therefore, using the factors of $t$ in formula \eqref{ExpressionToEvaluate} together with the additional factors of $t$ present in the formulas \eqref{twisted-convolution} for the twisted convolution and definition \eqref{eq-smooth-supertrace} for the $t$-supertrace, the expression from formula \eqref{ExpressionToEvaluate} can be written as
\begin{equation} \label{ToEvaluate2}
  \int_{\Delta_k} \Str_t (H_{\tau_1} * K_{\tau_2-\tau_1}^1* \cdots * K_{1-\tau_N}^N) d \tau,
\end{equation}
where $*$ denotes the twisted convolution. By Lemma~\ref{LemmaExtensionRapidDecay}~(2), the integrand $H_{\tau_1} * K_{\tau_2-\tau_1}^1* \cdots * K_{1-\tau_N}^N$ is now an element of $\mathcal{S}(\T M, \mathbb{S})$, hence it can be evaluated at $t=0$. But if $k < N$, then necessarily one of the $K^p_\tau$ contains a factor of $F(\theta_{i_p-1}, \theta_{i_p})$, which evaluates to zero over $TM \times \{0\}$, by Lemma~\ref{LemmaGetzlerSymbolsF}. This shows the claim for $k < N$.

It is left to consider the case $k=N$. In this case, it follows from Lemma~\ref{LemmaGetzlerSymbolsF} and Prop.~\ref{PropHeatKernel} that over $t=0$, 
\begin{equation*}
\begin{aligned}
  H_{\tau_1} *  K_{\tau_2-\tau_1}^1* \cdots * K_{1-\tau_N}^N &= H_{\tau_1} \wedge (\theta_1^{\prime\prime} \wedge H_{\tau_2-\tau_1}) * \cdots * (\theta_N^{\prime\prime} \wedge H_{1-\tau_N}) \\
  &=  \theta_1^{\prime\prime} \wedge \cdots \wedge \theta_N^{\prime\prime} \wedge (H_{\tau_1} * H_{\tau_2-\tau_1} * \cdots  * H_{1-\tau_N}).
\end{aligned}
\end{equation*}
Here in the second step we used that the $\theta_i$ can be pulled out since they are constant as functions on $T_m M$ and the $H_\tau$ are even. The twisted convolution identity of $H_\tau$ (see Remark~\ref{RemarkTwistedConvIdentity}) now implies that
\begin{equation*} 
H_{\tau_1} * H_{\tau_2-\tau_1} * \cdots  * H_{1-\tau_N} = H_1.
\end{equation*}
By continuity of the $t$-supertraces (compare Lemma~\ref{LemmaExtensionRapidDecay} (3)), the expression \eqref{ToEvaluate2} is continuous in $t$. With a view on \eqref{eq-smooth-supertrace},  evaluating at $t=0$ yields 
\begin{equation*}
\begin{aligned}
\left(\frac{2}{i}\right)^{d/2}\int_{\Delta_N} \left(\int_M \theta_0^\prime \wedge \theta_1^{\prime\prime} \wedge \cdots \wedge \theta_N^{\prime\prime} \wedge H_1(0, -, 0)\right) d \tau
\end{aligned}
\end{equation*}
The integrand of the integral over $\Delta_N$ is independent of $\tau$, hence it just contributes a factor of $\mathrm{vol}(\Delta_N) = 1/N!$. Finally, comparing formula \eqref{heat-element-2} with \eqref{AHatForm}, one observes that $H_1(0, -, 0)$ is precisely $(4 \pi)^{-d/2}$ times the $\hat{A}$-form on $M$. In total, we obtain precisely the right hand side of \eqref{ChernLimitFormula}.
\end{proof}

	\bibliography{Refs} 
	\bibliographystyle{alpha}
	
	
	

\end{document}